\documentclass[11pt]{amsart}

\usepackage{amssymb}
\usepackage{amsmath,amsthm,amscd,amsfonts}
\usepackage{enumerate}

\usepackage[dvips, dvipsnames, usenames]{color}

\newcommand{\trid}{\triangleright}
\newcommand{\Aut}{\operatorname{Aut}}
\newcommand{\Inn}{\operatorname{Inn}}
\newcommand{\Out}{\operatorname{Out}}

\newcommand\id{\operatorname{id}}

\newcommand{\Sym}{\mathbb{S}}
\newcommand{\gap}{\textsf{GAP}}

\newcommand{\Z}{{\mathbb Z}}

\newcommand{\oc}{{\mathcal O}}
\newtheorem{step-j3}{Step}
\newtheorem{step-j4}{Step}
\newtheorem{step-ly}{Step}
\newtheorem{step-b}{Step}
\newtheorem{step-on}{Step}
\newtheorem{step-mcl}{Step}
\newtheorem{step-hn}{Step}
\newtheorem{step-fi24'}{Step}

\theoremstyle{plain}
\newtheorem{thm}{Theorem}[section]
\newtheorem{lem}[thm]{Lemma}
\newtheorem{rem}[thm]{Remark}
\newtheorem{pro}[thm]{Proposition}
\newtheorem{cor}[thm]{Corollary}

\newtheorem*{thm*}{Theorem}

\newcommand\toba{{\mathfrak B }}

\begin{document}

\thispagestyle{empty}

\title[Twisted conjugacy classes of type D in sporadic simple groups]
{On twisted conjugacy classes \\ of type D in sporadic simple groups}

\author[Fantino and Vendramin]{F. Fantino and L. Vendramin}

\address{FF: Facultad de Matem\'atica, Astronom\'{\i}a y F\'{\i}sica,
Universidad Nacional de C\'ordoba. CIEM -- CONICET. 
Medina Allende s/n (5000) Ciudad Universitaria, C\'ordoba,
Argentina}
\address{FF: UFR de Math\'ematiques, Universit\'e Paris Diderot - Paris 7, 175 Rue du Chevaleret, 75013, Paris, France}
\email{fantino@famaf.unc.edu.ar}
\urladdr{http://www.mate.uncor.edu/~fantino/}

\address{LV: Departamento de Matem\'atica, FCEN,
Universidad de Buenos Aires, Pab. I, Ciudad Universitaria (1428),
Buenos Aires, Argentina}
\email{lvendramin@dm.uba.ar}
\urladdr{http://mate.dm.uba.ar/~lvendram/}

\thanks{This work was partially supported by CONICET, ANPCyT-Foncyt, Secyt-UNC,
Embajada de Francia en Argentina}


\begin{abstract}
We determine twisted conjugacy classes of type D associated with the sporadic
simple groups. This is an important step in the program of the classification
of finite-dimensional pointed Hopf algebras with non-abelian coradical.  As a
by-product we prove that every complex finite-dimensional pointed Hopf algebra
over the group of automorphisms of $M_{12}$, $J_2$, $Suz$, $He$, $HN$, $T$ is
the group algebra. In the appendix we improve the study of conjugacy
classes of type D of sporadic simple groups.
\end{abstract}

\maketitle

\section{Introduction}

A fundamental step in the classification of finite-dimensional complex
point\-ed Hopf algebras, in the context of the Lifting method \cite{MR1659895},
is the determination of all finite-dimensional Nichols algebras of braided
vector spaces arising from Yetter-Drinfled modules over groups. This problem
can be reformulated in other terms: to study finite-dimensional Nichols
algebras of braided vector spaces arising from pairs $(X,q)$, where $X$ is a
rack and $q$ is a 2-cocycle of $X$.

A useful strategy to deal with this problem is to discard those pairs $(X,q)$
whose associated Nichols algebra is infinite dimensional. A powerful tool to
discard such pairs is the notion of rack of type D \cite{AMPA}. This notion is
based on the theory of Weyl groupoids developed in \cite{MR2766176} and
\cite{MR2734956}.  The importance of racks of type D lies in the following
fact: if $X$ is a rack of type D then the Nichols algebra associated to $(X,q)$
is infinite-dimensional for all 2-cocycle $q$.
The property of being of type D is well-behaved with respect to monomorphisms and
epimorphisms, see Remark \ref{rem:tipoD:inj:proy}.
On the other hand, it is well-known
that a finite rack can be decomposed as an union of indecomposable subracks.
Further, every indecomposable rack $X$ admits a surjection $X\to Y$, where $Y$
is a simple rack, and the classification of finite simple racks is known, see
\cite{MR1994219} and \cite{MR682881}.
These facts and the ubiquity of racks of type D
suggests a powerful approach for the classification problem of
finite-dimensional pointed Hopf algebras over non-abelian groups: to classify
finite simple racks of type D.  This program was described in \cite[\S 2]{CLA}
and successfully applied to the classification of finite-dimensional pointed
Hopf algebras over the alternating simple groups \cite{AMPA} and over many of
the sporadic simple groups \cite{MR2745542}.  This paper is a contribution to
this program.

Towards the classification of simple racks of type D, we study an important
family of simple racks: the twisted conjugacy classes of a sporadic simple
group $L$.  Our aim is to classify which of these racks are of type D. For that
purpose, we use the fact that these racks can be realized as conjugacy classes
of the group of automorphisms of $L$.  The main result of our work is the
following theorem.

\begin{thm}
\label{thm:main}
Let $L$ be one of the simple groups
\[
M_{12},\; M_{22},\; J_2,\; J_3, \; Suz,\; HS,\; McL,\; He,\; Fi_{22},\; ON,\; Fi_{24}', \; HN,\; T.
\]
Let $\mathcal{O}$ be a conjugacy class of
$\Aut(L)$ not contained in $L$ which is not listed in
Table \ref{tab:classes}.
Then $\mathcal{O}$ is of type D.
\end{thm}

\begin{table}[h]
\begin{center}
\caption{Classes not of type D}
\label{tab:classes}
\begin{tabular}{|c|c|}
\hline
$\mathrm{Aut}(M_{22})$ & $\textup{2B}$\tabularnewline
$\mathrm{Aut}(HS)$ & $\textup{2C}$\tabularnewline
$\mathrm{Aut}(Fi_{22})$ & $\textup{2D}$\tabularnewline
$\mathrm{Aut}(J_{3})$ & $\textup{34A, 34B}$\tabularnewline
$\mathrm{Aut}(ON)$ & $\textup{38A, 38B, 38C}$\tabularnewline
$\mathrm{Aut}(McL)$ & $\textup{22A, 22B}$\tabularnewline
$\mathrm{Aut}(Fi_{24}')$ & $\textup{2C}$\tabularnewline
\hline
\end{tabular}
\end{center}
\end{table}

Notice that the groups in Theorem \ref{thm:main} are the only sporadic simple
groups with non-trivial outer automorphism group.
Theorem \ref{thm:main} with \cite{MR2745542} and the lifting method
\cite{MR1659895} imply the following classification result.

\begin{cor}
\label{cor:classification}
Let $L$ be one of the simple groups
\[
M_{12},\; J_2,\; Suz,\; He,\; HN,\; T.
\]
Then $\Aut(L)$ does not have non-trivial finite-dimensional complex pointed
Hopf algebras.\qed
\end{cor}

The study of twisted conjugacy classes of sporadic groups is suitable for
being attacked case-by-case with the help of computer calculations.
The strategy for proving Theorem \ref{thm:main} is the same as in
\cite{MR2745541, MR2745542}.  We use the computer algebra system \gap~ to
perform the computations \cite{GAP} \cite{B} \cite{AtlasRep1.4}
\cite{ATLASwww}.
The main scripts and log files of this work can be found in:
\verb+http://www.famaf.unc.edu.ar/~fantino/fv.tar.gz+
or
\verb+http://mate.dm.uba.ar/~lvendram/fv.tar.gz+.

\medskip
The paper is organized as follows. In Section \ref{se:preliminaries} we deal
with the basic definitions and the basic techniques for studying Nichols
algebras over simple racks. The proof of the main result is given in Section
\ref{se:mainresult}.  In the appendix we improve the classification of racks of
type D given in \cite[Table 2]{MR2745542}.  With the exception of the Monster
group, conjugacy classes of type D in sporadic simple groups are classified,
see Remark \ref{rem:M}.

\section{Preliminaries}
\label{se:preliminaries}

We refer to \cite{MR1913436} for generalities about Nichols algebras and to
\cite{MR1994219} for generalities about racks and their cohomologies in the
context of Nichols algebras.  We follow \cite{MR827219} for the notations
concerning the sporadic simple groups.

\medskip
A \textit{rack} is a pair $(X,\trid)$, where $X$ is a non-empty set and
$\trid:X\times X\to X$ is a map (considered as a binary operation on $X$) such
that the map $\varphi_x:X\to X$, $\varphi_x(y)=x\trid y$, is bijective for all
$x\in X$, and $x\trid(y\trid z)=(x\trid y)\trid(x\trid z)$ for all $x,y,z\in
X$.  A \textit{subrack} of a rack $X$ is a non-empty subset $Y\subseteq X$ such
that $(Y,\trid)$ is also a rack. A rack $X$ is said to be a \emph{quandle} if
$x\trid x=x$ for all $x\in X$. All the racks considered in this work are indeed
quandles.

\medskip
A rack $(X,\trid)$ is said to be of \emph{type D} if it contains a decomposable
subrack $Y = R\sqcup S$ such that $r\trid(s\trid(r\trid s)) \neq s$ for some
$r\in R$, $s\in S$.

\begin{rem}
Let $G$ be a group. A conjugacy class $\mathcal{O}$ of $G$ is of type D if and
only if there exist $r,s\in\mathcal{O}$ such that $(rs)^2\ne (sr)^2$ and $r$
and $s$ are not conjugate in the group generated by $r$ and $s$, see
\cite[Subsection 2.2]{MR2745542}.
\end{rem}

\begin{rem}
Racks of type D have the following properties:
\label{rem:tipoD:inj:proy}
\begin{enumerate}[(i)]
\item If $Y \subseteq X$ is a subrack of type D, then $X$ is of type D.
\item If $Z$ is a finite rack and $p:Z\to X$ is an epimorphism, then $X$ of
    type D implies $Z$ of type D.
\end{enumerate}
\end{rem}

The following result is the reason why it is important to study racks of type
D. This theorem is based on  \cite{MR2766176} and \cite{MR2734956}.

\begin{thm}\cite[Thm. 3.6]{AMPA}
\label{thm:AMPA}
Let $X$ be a finite rack of type D. Then the Nichols algebra associated with
the pair $(X,q)$ is infinite-dimensional for all $2$-cocycles $q$. \qed
\end{thm}

A rack is \emph{simple} if it has no quotients except itself and the
one-element rack.  We recall the classification of finite simple racks given in
\cite[Theorems 3.9 and 3.12]{MR1994219}, see also \cite{MR682881}.  A finite
simple rack belongs to one of the following classes:
\begin{enumerate}
\renewcommand{\theenumi}{\alph{enumi}}\renewcommand{\labelenumi}{(\theenumi)}
\item \label{item:sar} simple affine racks;
\item non-trivial conjugacy classes of non-abelian finite simple groups;
\item \label{item:tcc} non-trivial twisted conjugacy classes of non-abelian finite simple groups;
\item \label{item:thr} simple twisted homogeneous racks.
\end{enumerate}

In this paper we study non-trivial twisted conjugacy classes of type D of
sporadic simple groups. These racks belong to the class \eqref{item:tcc}
mentioned above.

\subsection{Twisted conjugacy classes}
\label{subsec:TCC}

Let $G$ be a finite group and $u\in \Aut (G)$. The group $G$ acts on itself by
$y\rightharpoonup_u x = y\,x\,u(y^{-1})$ for all $x,y\in G$.  The orbit of $x$
under this action will be called the \emph{$u$-twisted conjugacy class} of $x$
and it will be denoted by $\oc^{G,u}_x$.
It is easy to prove that the orbit $\oc^{G,u}_x$ is a rack with
\begin{equation*}\label{eqn:twisted-homogeneous}
y\trid_u z = y\,u(z\,y^{-1})
\end{equation*}
for all $y,z\in \oc^{G,u}_x$.
Notice that $\oc_x^{G,\id}$ is a conjugacy class in G.

We write $\Out(G):=\Aut(G)/\Inn(G)$ for the group of outer automorphisms of $G$
and $\pi:\Aut(G)\to\Out(G)$ for the canonical surjection.

Assume that $\Out(G)\ne1$. Let $u\in\Aut(G)$ such that $\pi(u)\ne1$.
Every $u$-twisted conjugacy class in $G$ is isomorphic (as a rack) to a
conjugacy class in the semidirect product
$G\rtimes\langle u\rangle$.
Indeed,
\begin{align*}
\label{eq:CC:TCC}
\oc_{(x,u)}^{G\rtimes \langle u \rangle,\mathrm{id}} = \oc^{G,u}_x \times \{u\}
\end{align*}
for all $x\in G$.
Therefore the problem of determining $u$-twisted conjugacy classes of type D in
$G$ can be reduced to study conjugacy classes of type D in $G\rtimes \langle u
\rangle$ and contained in $G\times \{u \}$.

\subsection{Conjugacy classes to study}
\label{tostudy}
Let $L$ be one of the simple groups
\[
M_{12},\; M_{22},\; J_2,\; J_3, \; Suz,\; HS,\; McL,\; He,\; Fi_{22},\; ON,\; Fi_{24}', \; HN,\; T.
\]
It is well-known that $\Aut(L)\simeq L\rtimes\Z_2$ \cite{MR827219}. Hence,
since $L$ is a normal subgroup of $L\rtimes\Z_2$, it is possible to compute the
list of conjugacy classes of $\Aut(L)$ not contained in $L$ from the character
table of $\Aut(L)$, see for example \cite{MR2270898}. For that purpose, we use
the \gap~function \verb+ClassPositionsOfDerivedSubgroup+. See the file
\verb+classes.log+ for the information concerning the conjugacy classes of
$\Aut(L)$ not contained in $L$.

\subsection{Strategy}
Our aim is to classify twisted conjugacy classes of sporadic simple groups of
type D.  By Subsection \ref{subsec:TCC}, we need to consider the conjugacy
classes in $\Aut(L)\setminus L$, where $L$ is a sporadic simple group with
$\Out(L)\ne1$. The strategy for studying these conjugacy classes is essentially
based on studying conjugacy classes of type D in maximal subgroups of
$\Aut(L)$. See \cite[Subsection 1.1]{MR2745541} for an exposition about
the algorithms used.

\subsection{Useful lemmata}
Let $G$ be a non-abelian group and $g\in G$. We write $g^G$ for the conjugacy class of $g$ in $G$. Let
\[
\mathfrak{M}_{g}=\{M:M\text{ is a maximal subgroup of \ensuremath{G} and \ensuremath{g^{G}\cap M\ne\emptyset\}}}.
\]

\begin{lem}[Breuer]
\label{lem:Breuer_general}Assume that for all $M\in\mathfrak{M}_{g}$
there exists $m\in M$ such that
\[
g^{G}\cap M\subseteq m^{M}\subseteq g^{G}.
\]
If $g^{G}$ is of type D, then there exist $N\in\mathfrak{M}_{g}$
and $n\in N$ such that $n{}^{N}$ is of type D.
\end{lem}

\begin{proof}
Since $g^{G}$ is of type D, there exist $r,s\in g^{G}$ such that
$(rs)^{2}\ne(sr)^{2}$ and $r^{H}\cap s^{H}=\emptyset$ for $H=\langle
r,s\rangle$.  Since $r^{H}\cap s^{H}=\emptyset$ and $r^{H}\cup s^{H}\subseteq
g^{G}$, the group $H$ is contained in some maximal subgroup
$N\in\mathfrak{M}_{g}$.  Hence $r,s\in g^{G}\cap N\subseteq n^{N}$ for some
$n\in N$ and the claim follows.
\end{proof}

%

\begin{lem}
\label{lem:usefulresult}
	Let $\mathcal{O}$ be a conjugacy class of $G$ and let $H$ be a subgroup
	of $G$ such that $\mathcal{O}$ contains two conjugacy classes $\oc_1$,
	$\oc_2$ of $H$. Assume that there exist $r\in\oc_1$ and $s\in\oc_2$
	such that $(rs)^2$ does not belong to the centralizer of $r$ in $G$.
	Then $\mathcal{O}$ is of type D.
\end{lem}

\begin{proof}
Notice that $(rs)^2=(sr)^2$ if and only if $(rs)^2$ commutes with $r$. Then the
claim follows.
\end{proof}

\begin{lem}\label{prop:tipoD:inv}
Let $\oc$ be a conjugacy class of involutions of $G$. Then $\oc$ is of type D
if and only if there exist $r$, $s\in \oc$ such that the order of $rs$ is even
and greater or equal to 6.
\end{lem}

\begin{proof}
Assume that $|rs|=n$. Then $\langle r,s\rangle\simeq\mathbb{D}_n$ and the claim
follows; see \cite[\S 1.8]{MR2745542}.
\end{proof}

\section{Proof of Theorem \ref{thm:main}}
\label{sec:TTC:espo}
\label{se:mainresult}

The claim concerning the automorphism groups of $M_{12}$ and $J_2$ follows from
the application of \cite[Algorithm I]{MR2745542}. The claim for the
automorphism groups of $M_{22}$, $Suz$, $HS$, $He$, $Fi_{22}$
and $T$ follows from the application of \cite[Algorithm III]{MR2745542}.
There is one log file for each of these groups, see Table \ref{tab:logs}.  The
automorphism groups of $J_3$, $ON$, $McL$, $HN$ and $Fi_{24}'$ are studied in Subsections
\ref{subsec:ON_McL_J3}, \ref{subsec:HN} and
\ref{subsec:Fi24'}, respectively.

\begin{table}[ht]
\begin{center}
\caption{Log files}\label{tab:logs}
\begin{tabular}{|l|l|l|l|}
\hline
$L$ & log file & $L$ & log file\tabularnewline
\hline
$M_{12}$ & \verb+M12.2.log+ & $HS$ & \verb+HS.2.log+\tabularnewline
$M_{22}$ & \verb+M22.2.log+ & $He$ & \verb+He.2.log+\tabularnewline
$J_{2}$ & \verb+J2.2.log+ & $Fi_{22}$ & \verb+Fi22.2.log+\tabularnewline
$Suz$ & \verb+Suz.2.log+ & $T$ & \verb+T.2.log+\tabularnewline
\hline
\end{tabular}
\end{center}
\end{table}

\subsection{The groups $\Aut(J_3)$, $\Aut(ON)$ and $\Aut(McL)$}
\label{subsec:ON_McL_J3}

\begin{lem}\
	\begin{enumerate}
		\item A conjugacy class $\mathcal{O}$ of $\Aut(J_3)\setminus J_3$ is of
			type D if and only if $\mathcal{O}\notin\{\textup{34A, 34B}\}$.
		\item A conjugacy class $\mathcal{O}$ of $\Aut(ON)\setminus ON$ is of type
			D if and only if $\mathcal{O}\notin\{\textup{38A, 38B, 38C}\}$.
		\item A conjugacy class $\mathcal{O}$ of $\Aut(McL)\setminus McL$ is of
			type D if and only if $\mathcal{O}\notin\{\textup{22A, 22B}\}$.
	\end{enumerate}
\end{lem}

\begin{proof}
	We first prove (1). We claim that the classes \textup{34A, 34B} are not of
	type D.  Let $G=\Aut(J_3)$ and let $g$ be a representative of the conjugacy
	class 34A (the proof for the class 34B is analogous). By \cite{MR827219}, the
	only maximal subgroup containing elements of order 34 is
	$\mathcal{M}_{4}\simeq\mathbf{PSL}(2,17)\times\mathbb{Z}_{2}$.  Further, it
	is easy to see that $\mathcal{M}_{4}\in\mathfrak{M}_{g}$ satisfies $g^{G}\cap
	\mathcal{M}_{4}\subseteq m^{\mathcal{M}_{4}}\subseteq g^{G}$ for some $m\in
	\mathcal{M}_{4}$ and the class $m^{\mathcal{M}_{4}}$ is not of type D. Hence
	Lemma \ref{lem:Breuer_general} applies. See the file \verb+J3.2/34AB.log+ for
	more information.  To prove that the remaining conjugacy classes are of type
	D, apply \cite[Algorithm III]{MR2745542}. See the file \texttt{J3.2/J3.2.log}
	for more information.

	Now we prove (2). We claim that the classes \textup{38A, 38B, 38C} of
	$\Aut(ON)$ are not of type D.  The only maximal subgroup (up to conjugation)
	of $\Aut(ON)$ containing elements of order 38 is the second maximal subgroup
	$\mathcal{M}_2$.  By Lemma \ref{lem:Breuer_general}, it suffices to prove
	that the classes 38a, 38b, 38c of $\mathcal{M}_2$ are not of type D.  This
	follows from a direct \gap~computation.  See the file \verb+ON.2/38ABC.log+
	for more information.  To prove that the remaining conjugacy classes are of
	type D we apply \cite[Algorithm III]{MR2745542}. See the file
	\texttt{ON.2/ON.2.log} for more information.

	Now we prove (3). We claim that the classes \textup{22A, 22B} of $\Aut(McL)$
	are not of type D.  The only maximal subgroup (up to conjugation) of
	$\Aut(McL)$ containing elements of order 22 is the 8th maximal subgroup
	$\mathcal{M}_8$.  By Lemma \ref{lem:Breuer_general}, it suffices to prove
	that the classes 22a, 22b of $\mathcal{M}_8$ are not of type D. This follows
	from a direct \gap~computation.  See the file \verb+McL.2/22AB.log+ for more
	information.  To prove that the remaining conjugacy classes are of type D,
	apply \cite[Algorithm III]{MR2745542}. See the file \texttt{McL.2/McL.2.log}
	for more information.
\end{proof}

\subsection{The group $\Aut(HN)$}
\label{subsec:HN}

\begin{lem}
\label{lem:HN}
All the conjugacy classes in $\Aut(HN)\setminus HN$ are of type D.
\end{lem}

\begin{proof}
With \gap~ it is possible to obtain the information related to the fusion of
the conjugacy classes from the maximal subgroups of $\Aut(HN)$ into $\Aut(HN)$.
The following table shows the
maximal subgroup (and the log file) used and the conjugacy classes of $\Aut(HN)$
of type D.
\begin{center}
\begin{tabular}{|c|c|}
	\hline
 File & Classes \tabularnewline
 \hline
 $\verb+M2+$ & \textup{4D,4E,4F,6D,6E,6F,8C,8D,10G,10H,12D} \tabularnewline
 & \textup{12E,14B,18A,20F,24A,28A,30C,42A,60A} \tabularnewline
 \hline
 $\verb+M13+$ & \textup{8F,24B,24C} \tabularnewline
 \hline
 $\verb+M9+$ & \textup{8E} \tabularnewline
 \hline
 $\verb+M7+$ & \textup{20E,20G,20H,20I,40B,40C,40D}\tabularnewline
 \hline
 $\verb+M3+$ & \textup{44A,44B}\tabularnewline
 \hline
\end{tabular}
\end{center}

It remains to prove that the class \textup{2C} is of type D.  By \cite[Thm.
4.1]{AMPA}, the class of transpositions in $\Sym_{12}$ is the unique class of
involutions which is not of type D.  But there are three different conjugacy
classes of involutions of the maximal subgroup $\mathcal{M}_2\simeq\Sym_{12}$
contained in the class \textup{2C} of $\Aut(HN)$ and hence the latter is of
type D.
%
%
%
%
%
%
%
%
%
%
\end{proof}

\subsection{The group $\Aut(Fi_{24}')$}
\label{subsec:Fi24'}

\begin{lem}
\label{lem:Fi24'}
Let $\mathcal{O}$ be a conjugacy class of $\Aut(Fi_{24}')\setminus Fi_{24}'$.
Then $\mathcal{O}$ is of type D if and only if
$\mathcal{O}\notin\{\textup{2C}\}$.
\end{lem}

\begin{proof}
As before, we study conjugacy classes in certain maximal subgroups. See the following table
for more information:
\begin{center}
\begin{tabular}{|c|c|}
	\hline
	File & Classes \tabularnewline
	\hline
	\verb+M5+ & \textup{4G, 12N, 12U, 12V, 12Y, 24H, 40A} \tabularnewline
	\hline
	\verb+M17+ & \textup{30G}\tabularnewline
	\hline
	\verb+M18+ & \textup{18O} \tabularnewline
	\hline
	\verb+M9+ &  \textup{66A, 66B} \tabularnewline
	\hline
	\verb+M12+ & \textup{36D, 36G} \tabularnewline
	\hline
	\verb+M20+ & \textup{12A1, 28C, 28D} \tabularnewline
	\hline
\end{tabular}
\end{center}

To prove that the classes \textup{6V, 42D, 84A} are of type D we use the
maximal subgroup $\mathcal{M}_{19}\simeq(\Z_7\rtimes\Z_6)\times\Sym_7$.  Notice
that six conjugacy classes of $\mathcal{M}_{19}$ are contained in the class
\textup{6V}.  Further, three of them are of type D.  On the other hand, the
conjugacy classes of elements of order 42 and 84 in $\mathcal{M}_{19}$ are of
type D.

The class \textup{2D} of $\Aut(Fi_{24}')$ contains the classes
$\mathcal{O}_1=\textup{2d}$ and $\mathcal{O}_2=\textup{2g}$ of the maximal
subgroup $\mathcal{M}_4\simeq \mathbb S_3\times O_8^+(3).\mathbb S_3$.
With \gap~we show that there exist $r\in\mathcal{O}_1$ and $s\in\mathcal{O}_2$
such that $(rs)^2$ has order $13$ which does not divide 
order of the centralizer of the conjugacy class \textup{2A} of
$\Aut(Fi_{24}')$.  Hence $(rs)^2$ does not commute with $r$ and therefore Lemma
\ref{lem:usefulresult} applies and hence \textup{2D} is of type D.

The class \textup{2C} is not of type D.  Indeed, for
all $r,s\in \textup{2C}$, the order of $rs$ is 1, 2 or 3; for this we use the
\textsf{GAP} function \texttt{ClassMultiplicationCoefficient}. By Lemma
\ref{prop:tipoD:inv}, the claim holds.

For studying the remaining conjugacy classes we use the maximal subgroup
$\mathcal{M}_2\simeq \mathbb{Z}_2\times Fi_{23}$.  By \cite[Thm.
II]{MR2745542}, every conjugacy class of $\mathcal{M}_2$ with representative of
order distinct from 2
is of type D, see Proposition \ref{pro:Fi23} in the appendix. Hence the claim
follows.
\end{proof}

\section*{Appendix: the sporadic simple groups}

\begin{table}[t]
\caption{Classes in sporadic simple groups not of type D}
\begin{center}
\label{tab:notD}
\begin{tabular}{|c|c|}
\hline Group & Classes\\
\hline $T$ &  \textup{2A}\\
\hline $M_{11}$ & \textup{8A, 8B, 11A, 11B}\\
\hline $M_{12}$ &  \textup{11A, 11B}\\
\hline $M_{22}$ &  \textup{11A, 11B} \\
\hline $M_{23}$ &  \textup{23A, 23B}\\
\hline $M_{24}$ &  \textup{23A, 23B}\\
\hline $Ru$ &  \textup{29A, 29B}\\
\hline $Suz$ &  \textup{3A}\\
\hline $HS$ &  \textup{11A, 11B}\\
\hline $McL$ &  \textup{11A, 11B}\\
\hline $Co_{1}$ &  \textup{3A} \\
\hline $Co_{2}$ &  \textup{2A, 23A, 23B} \\
\hline $Co_{3}$ &  \textup{23A, 23B} \\
\hline $J_{1}$ &  \textup{15A, 15B, 19A, 19B, 19C}\\
\hline $J_{2}$ &  \textup{2A, 3A} \\
\hline $J_{3}$ &  \textup{5A, 5B, 19A, 19B}\\
\hline $J_4$ &  \textup{29A, 43A, 43B, 43C}\\
\hline $Ly$ &  \textup{37A, 37B, 67A, 67B, 67C}\\
\hline $O'N$ &  \textup{31A, 31B}\\
\hline $Fi_{23}$ & \textup{2A}\\
\hline $Fi_{22}$ & \textup{2A, 22A, 22B}\\
\hline $Fi'_{24}$ & \textup{29A, 29B}\\
\hline $B$ & \textup{2A, 46A, 46B, 47A, 47B}\\
\hline
\end{tabular}
\end{center}
\end{table}
In this appendix we improve some of the results obtained in \cite{MR2745542}.
We remark that it is important to know if a rack $X$ is of type D.  By Theorem
\ref{thm:AMPA}, if $X$  is of type D, then $\dim\toba(X,q)=\infty$ for any
2-cocycle $q$ of $X$, and
hence the calculation of the $2$-cocycles of $X$ is not needed.
Further, as a corollary we obtain that the Nichols algebras associated with any
rack $Y$ containing $X$ and for any rack $Z$ having $X$ as a quotient are also
infinite-dimensional; see \cite{CLA} or \cite{MR2745542}. Indeed, since any finite rack
has a projection onto a simple rack, this shows the intrinsic importance of a
conjugacy class of being of type D, not just for the specific group where it
lives but to the whole classification program of finite-dimensional Nichols
algebras associated to racks.

Table \ref{tab:notD} contains the list of conjugacy classes of sporadic simple
groups which are not of type D.  The open cases are listed in Remark
\ref{rem:M}.

\subsection{The groups $T$ and $Suz$}

\begin{pro}\
	\begin{enumerate}
		\item A conjugacy class $\mathcal{O}$ of $T$ is of type D if and only if
			$\mathcal{O}\ne\textup{2A}$.
		\item A conjugacy class $\mathcal{O}$ of $Suz$ is of type D if and only if
			$\mathcal{O}\ne\textup{3A}$.
	\end{enumerate}
\end{pro}

\begin{proof}
	It follows from \cite[Table 2]{MR2745542} and a direct computer calculation.
	See the log files for details.
\end{proof}

\pagebreak
\subsection{The groups $ON$, $McL$, $Co_3$, $Ru$, $HS$ and $J_3$}

\begin{pro}
\label{pro:ON}\
\begin{enumerate}
	\item A conjugacy class of $ON$ is of type D if and only if it is different
		from \textup{31A} and \textup{31B}.
	\item A conjugacy class of $McL$ is of type D if and only if it is different
		from \textup{11A} and \textup{11B}.
	\item A conjugacy class of $Co_3$ is of type D if and only if it is different
		from \textup{23A}, \textup{23B}.
	\item A conjugacy class of $Ru$ is of type D if and only if it is different
		from \textup{29A} and \textup{29B}.
	\item A conjugacy class of $HS$ is of type D if and only if it is different
		from \textup{11A} and \textup{11B}.
	\item A conjugacy class of $J_3$ is of type D if and only if it is different
		from \textup{5A}, \textup{5B}, \textup{19A} and \textup{19B}.
\end{enumerate}
\end{pro}

\begin{proof}
	We prove (1). By \cite[Table 2]{MR2745542}, it remains to prove that the
	classes \textup{31A, 31B} are not of type D. Let $g$ be a representative for
	the conjugacy class 31A of $G=ON$ (the proof for the class 31B is analogous).
	By \cite{MR827219}, the only maximal subgroups (up to conjugacy) containing
	elements of order 31 are $\mathcal{M}_{7}$ and $\mathcal{M}_{8}$.  Further,
	$\mathcal{M}_{7}\simeq \mathcal{M}_{8}\simeq\mathbf{PSL}(2,7)$ and it is easy
	to see that if $M=\mathcal{M}_{7}$ (or $\mathcal{M}_{8}$) then $g^{G}\cap
	M\subseteq m^{M}\subseteq g^{G}$ for some $m\in M$.  Since the conjugacy
	class $m^{M}$ is not of type D for all $m\in M$ of order 31, Lemma
	\ref{lem:Breuer_general} applies.

	To prove (2) the maximal subgroups to use are the Mathieu groups $M_{11}$ and
	$M_{22}$. Then the claim follows from Lemma \ref{lem:Breuer_general} and
	\cite[Table 2]{MR2745542}.  To prove (3) the maximal subgroups to use are the
	Mathieu groups $M_{23}$. The proofs for (4)--(6) are similar.
\end{proof}


\subsection{The group $Fi_{23}$}

\begin{pro}
\label{pro:Fi23}
Let $\mathcal{O}$ be a conjugacy class of $Fi_{23}$. Then $\mathcal{O}$ is of
type D if and only if $\mathcal{O}$ is not \textup{2A}.
\end{pro}

\begin{proof}
By \cite[Table 2]{MR2745542}, it remains to prove that the classes \textup{23A,
23B} are of type D.  Let $N$ denote the normal subgroup of order $2^{11}$ in
the maximal subgroup $\mathcal{M}_6\simeq 2^{11}.M_{23}$ of $Fi_{23}$, and let
$x$ be an element of order 23 in the factor group $\mathcal{M}_6/N$. All
preimages of $x$ under the natural epimorphism from $\mathcal{M}_6$ to
$\mathcal{M}_6/N$ have order 23, they are conjugate in $\mathcal{M}_6$, and
their squares are also conjugate in $\mathcal{M}_6$.  Take a preimage $r$ of
$x$ under the natural epimorphism, choose a nonidentity element $n \in N$, and
set $s = r^2 n$.  Then $r$ and $s$ are conjugate in $\mathcal{M}_6$. Moreover,
$(rs)^2$ and $(sr)^2 = r^{-1} (rs)^2 r$ are different. Indeed, $(rs)^2 = (r^3
n)^2 = r^6 n'$, with $n':= (r^{-3} n r^3) n$, whereas $(sr)^2 = r^6 (r^{-1} n'
r)$.

The group $U$ generated by $r$ and $s$ is also generated by $r$ and $n$, and
since $r$ acts irreducibly on $N$, we get that $U$ is a semidirect product of
$N$ and $\langle r \rangle $.  In particular, $r$ and $s$ are not conjugate in
$U$. Hence, the class of $r$ in $\mathcal{M}_6$ is of type D.
\end{proof}

\subsection{The groups $Fi_{22}$ and $Co_2$}

\subsubsection{The group $Fi_{22}$}

\begin{pro}[Breuer]
\label{pro:Fi22_22A22B}
Let $\mathcal{O}$ be a conjugacy class of $Fi_{22}$. Then $\mathcal{O}$ is of type D
if and only if $\mathcal{O}\notin\{\textup{2A, 22A, 22B}\}$.
\end{pro}

\begin{proof}
By \cite[Table 2]{MR2745542}, it remains to prove that the classes \textup{22A,
22B} are not of type D.  Assume that the class 22A of $Fi_{22}$ is of type D
(the proof for the class 22B is analogous). Let $r$ and $s$ be elements of the
class 22A such that $r$ and $s$ are not conjugate in the group $H=\langle
r,s\rangle$.  By the fusion of conjugacy classes, $H$ is a proper subgroup of
some maximal subgroup $M$ isomorphic to $2.U_{6}(2)$. Notice that the center of
$M$ is $Z(M)=\langle z\rangle\simeq\mathbb{Z}_{2}$. Using the \gap~function
\texttt{PowerMap} we get $r^{11}=s^{11}=z$ and hence $Z(M)\subseteq H$.  We
claim that the elements $rZ(M)$ and $sZ(M)$ are not conjugate in the quotient
$H/Z(M)$. Let $p:H\to H/Z(M)$ be the canonical projection, and let $x\in H$
such that $p(x)p(r)p(x)^{-1}=p(s)$. Then $xrx^{-1}\in\{s,sz\}$ and hence
$xrx^{-1}=s$ since $sz$ has order $11$. Now the claim follows from the
following lemma.
\end{proof}

\begin{lem}
Let $Q=U_{6}(2)$ and $x,y\in Q$ be two elements of order $11$ such
that $x^{Q}=y^{Q}$. Assume that $x$ and $y$ are not conjugate in
the subgroup $\langle x,y\rangle$. Then $\langle x,y\rangle\simeq\mathbb{Z}_{11}$.
\end{lem}

\begin{proof}
Let $U=\langle x,y\rangle$. Since $x^{Q}=y^{Q}$ and $x^{U}\ne y^{U}$, $U$ is a
proper subgroup of a maximal subgroup $M$ and $M\simeq U_{5}(2)$ or $M\simeq
M_{22}$. The only maximal subgroup of $M$ which contains elements of order $11$
is isomorphic to $L_{2}(11)$ and hence we may assume that $U$ is a proper
subgroup of $L_{2}(11)$ because $L_{2}(11)$ has exactly two conjugacy classes
of elements of order $11$. The only maximal subgroups of $L_{2}(11)$ that
contain elements of order $11$ are isomorphic to
$\mathbb{Z}_{11}\rtimes\mathbb{Z}_{5}$.  These groups have exactly two
conjugacy classes of elements of order $11$ and hence $U$ must be a proper
subgroup of $\mathbb{Z}_{11}\rtimes\mathbb{Z}_{5}$.  From this the claim
follows.
\end{proof}

\subsubsection{The Conway group $Co_{2}$}

\begin{pro}
Let $\mathcal{O}$ be a conjugacy class of $Co_2$. Then $\mathcal{O}$ is of type D
if and only if $\mathcal{O}\notin\{\textup{2A, 23A, 23B}\}$.
\end{pro}

\begin{proof}
By \cite[Table 2]{MR2745542}, it remains to prove that the classes \textup{23A,
23B} are not of type D. This follows from the following lemma.
\end{proof}

\begin{lem}
\label{lem:Co2}
Let $Q=Co_2$ and $x,y\in Q$ be two elements of order $23$ such that
$x^{Q}=y^{Q}$. Assume that $x$ and $y$ are not conjugate in the subgroup
$\langle x,y\rangle$. Then $\langle x,y\rangle\simeq\mathbb{Z}_{23}$.
\end{lem}

\begin{proof}
Let $U=\langle x,y\rangle$. Since $x^{Q}=y^{Q}$ and $x^{U}\ne y^{U}$, $U$ is a
proper subgroup of a maximal subgroup $M$ and $M\simeq M_{23}$. The only
maximal subgroup of $M$ which contains elements of order $23$ is isomorphic to
$\mathbb{Z}_{23}\rtimes\mathbb{Z}_{11}$.  These groups have exactly two
conjugacy classes of elements of order $23$ and hence $U$ must be a proper
subgroup of $\mathbb{Z}_{23}\rtimes\mathbb{Z}_{11}$. From this the claim
follows.
\end{proof}

\subsection{The groups $J_4$, $Ly$, $Fi_{24}'$ and $B$}

\subsubsection{The Janko group $J_4$}

\begin{pro}
Let $\mathcal{O}$ be a conjugacy class of $J_4$. Then $\mathcal{O}$ is of type
D if and only if $\mathcal{O}\notin\{\textup{29A, 43A, 43B, 43C}\}$.
\end{pro}

\begin{proof}
By \cite[Table 2]{MR2745542} it remains to study the classes 29A, 37A, 37B,
37C, 43A, 43B, 43C.  We split the proof into two steps.

\begin{step-j4}
The conjugacy classes \textup{29A, 43A, 43B, 43C} of $J_{4}$ are not of type D.
\end{step-j4}

This is similar to the proof of Proposition \ref{pro:ON}.  See the files in the
folder \verb+J4+ for more information.

\begin{step-j4}
The classes \textup{37A, 37B, 37C} of $J_{4}$ are of type D.
\end{step-j4}

The class 37A of $J_4$ contains the classes $\mathcal{O}_1=\textup{37a}$ and
$\mathcal{O}_2=\textup{37d}$ of the maximal subgroup $\mathcal{M}_5\simeq
U_3(11).2$.  With \gap~we show that there exist $r\in\mathcal{O}_1$ and
$s\in\mathcal{O}_2$ such that $(rs)^2$ has order 5.  The centralizer associated
with the conjugacy class \textup{37A} of $J_4$ is isomorphic to
$\mathbb{Z}_{37}$ and therefore $(rs)^2$ does not commute with $r$.  Hence
Lemma \ref{lem:usefulresult} applies and the claim follows. The proof for the
classes \textup{37B, 37C} of $J_4$ is analogous, see the file
\verb+J4/37ABC.log+ for more information.
\end{proof}

\subsubsection{The group $Ly$}

\begin{pro}
Let $\mathcal{O}$ be a conjugacy class of $Ly$. Then $\mathcal{O}$ is of type
D if and only if $\mathcal{O}\notin\{\textup{37A, 37B, 67A, 67B, 67C}\}$.
\end{pro}

\begin{proof}
By \cite[Table 2]{MR2745542}, it remains to study the classes 33A, 33B, 37A,
37B, 37C, 37C, 67A, 67B, 67C.  We split the proof into two steps.

\begin{step-ly}
The classes \textup{37A, 37B, 67A, 67B, 67C} of $Ly$ are not of type D.
\end{step-ly}

It is similar to the proof of Proposition \ref{pro:ON}.  See the files
\verb+Ly/37AB.log+ and \verb+Ly/67ABC.log+ for more information.

\begin{step-ly}
The classes \textup{33A, 33B} of $Ly$ are of type D.
\end{step-ly}

It suffices to prove that the classes 33A and 33B of the maximal subgroup
$3.McL.2$ are of type D.  This follows from Lemma \ref{lem:usefulresult} with
the subgroup $3.McL$.  See the file \verb+Ly/33AB.log+ for more information.
\end{proof}

\subsubsection{The group $Fi_{24}'$}

\begin{pro}\
			A conjugacy class $\mathcal{O}$ of $Fi_{24}'$ is of type D if and only if
			$\mathcal{O}\notin\{\textup{29A, 29B}\}$.
\end{pro}

\begin{proof}
By \cite[Table 2]{MR2745542}, it remains to prove that the classes \textup{23A, 23B, 27B, 27C, 33A,
	33B, 39C, 39D} of $Fi_{24}'$ are of type D and that the classes \textup{29A, 29B} are not. For the classes \textup{23A} and
	\textup{23B} the result follows from Proposition \ref{pro:Fi23}. The following six classes can be treated by Lemma \ref{lem:usefulresult}. The
	 table below contains the information concerning the maximal subgroups to
	use:
	\begin{center}
  \begin{tabular}{|c|c|c|}
  	\hline
  	Classes & Maximal subgroup & Log file\tabularnewline
  	\hline
		$\textup{27B,27C}$ & $\mathcal{M}_{5}$ & \texttt{F3+/27BC.log}\tabularnewline
  	\hline
		$\textup{\textup{33A,33B}}$ & $\mathcal{M}_{4}$ & \texttt{F3+/33AB.log}\tabularnewline
  	\hline
		$\textup{39C,39D}$ & $\mathcal{M}_{3}$ & \texttt{F3+/39CD.log}\tabularnewline
  	\hline
  \end{tabular}
	\end{center}

	Now we prove that the classes \textup{29A, 29B} of $Fi_{24}'$ are not of type D.
	The unique maximal subgroup (up to conjugacy) that contains elements of order
	29 is $\mathcal{M}_{25}\simeq\mathbb{Z}_{29}\rtimes\mathbb{Z}_{14}$.  This
	group has two classes of elements of order 29 and these classes are not of
	type D.  Therefore Lemma \ref{lem:Breuer_general} applies.
\end{proof}

\subsection{The Conway group $Co_{1}$}

\begin{pro}
\label{pro:Co1}
Let $\mathcal{O}$ be a conjugacy class of $Co_1$. Then $\mathcal{O}$ is of type
D if and only if $\mathcal{O}\notin\{\textup{3A}\}$.
\end{pro}

\begin{proof}
By \cite[Table 2]{MR2745542}, it remains to prove that the classes \textup{23A,
23B} are of type D and that the class \textup{3A} is not of type D.  The proof
for the classes \textup{23A} and \textup{23B} is analogous to the proof of
Proposition \ref{pro:Fi23} using the maximal subgroup $\mathcal{M}_3\simeq
2^{11}:M_{24}$.  The claim for the class \textup{3A} follows from
a straightforward computer calculation.
\end{proof}

\subsubsection{The group $B$}

\begin{pro}
Let $\mathcal{O}$ be a non-trivial conjugacy class of $B$. Then
$\mathcal{O}$ is of type D if and only if $\mathcal{O}\notin\{
\textup{2A, 46A, 46B, 47A, 47B}\}$.
\end{pro}

\begin{proof}
By \cite{MR2745542}, it remains to study the conjugacy classes \textup{2A, 16C,
16D, 32A, 32B, 32C, 32D, 34A,
46A, 46B, 47A, 47B}.  We split the proof into
several steps.

\begin{step-b}
The conjugacy class \textup{2A} is not of type D.
\end{step-b}

With the \textsf{GAP} function \texttt{ClassMultiplicationCoefficient} we
see that for all $r,s\in \textup{2A}$, $|rs|$ is 1, 2, 3 or 4. Then the
claim follows from Lemma \ref{prop:tipoD:inv}.

\begin{step-b}
The conjugacy classes \textup{46A, 46B} of $B$ are not of type D.
\end{step-b}

Assume that the class 46A of $B$ is of type D (the proof for the class 46B is
analogous). Let $r$ and $s$ be elements of the class 46A such that $r$ and $s$
are not conjugate in the group $H=\langle r,s\rangle$.  By the fusion of
conjugacy classes, $H$ is a proper subgroup of some maximal subgroup $M$
isomorphic to $2^{1+22}.Co_2$. Notice that the center of $M$ is $Z(M)=\langle
z\rangle\simeq\mathbb{Z}_{2}$. With the \gap~function \texttt{PowerMap} we get
$r^{23}=s^{23}=z$ and hence $Z(M)\subseteq H$.  We claim that the elements
$rZ(M)$ and $sZ(M)$ are not conjugate in the quotient $H/Z(M)$. Let $p:H\to
H/Z(M)$ be the canonical projection, and let $x\in H$ such that
$p(x)p(r)p(x)^{-1}=p(s)$. Then $xrx^{-1}\in\{s,sz\}$ and hence $xrx^{-1}=s$
since $sz$ has order $23$. Now the claim follows from Lemma
\ref{lem:Co2}.

\begin{step-b}
The conjugacy classes \textup{47A, 47B} of $B$ are not of type D.
\end{step-b}

It is easy to check that the only maximal subgroup of $B$ (up to conjugacy)
which contains elements of order 47 is
$\mathcal{M}_{30}\simeq\mathbb{Z}_{47}\rtimes\mathbb{Z}_{23}$. (This is the
only non-abelian group of order 1081.) This group has two conjugacy classes
of elements of order 47 and these classes are not of type D. Then the claim
follows from Lemma \ref{lem:Breuer_general}.

\begin{step-b}
The class \textup{34A} of $B$ is of type D.
\end{step-b}

The conjugacy classes $\mathcal{O}_1=\textup{34d}$ and
$\mathcal{O}_2=\textup{34f}$ of the first maximal subgroup of $B$ are contained
in the class \textup{34A} of $B$.  With the \textsf{GAP} functions
\verb+ClassMultiplicationCoefficient+ and \verb+PowerMap+ we see
that there exist $r\in\mathcal{O}_1$  and $s\in\mathcal{O}_2$ such that
$|(rs)^2|=5$. Since the centralizer corresponding to the class \textup{34A} has
order 68, the claim follows from Lemma \ref{lem:usefulresult}.

\begin{step-b}
The classes \textup{16C, 16D} of $B$ are of type D.
\end{step-b}

Let $\mathcal{O}$ be the conjugacy class \textup{16C} (resp. \textup{16D}) of
$B$.  The conjugacy classes $\mathcal{O}_1=\textup{16g}$ (resp. \textup{16a})
and $\mathcal{O}_2=\textup{16n}$ (resp. \textup{16f}) of the first maximal
subgroup of $B$ are contained in the class $\mathcal{O}$.  With the
\textsf{GAP} functions \verb+ClassMultiplicationCoefficient+ and
\verb+PowerMap+ it is easy to see that there exist $r\in\mathcal{O}_1$ and
$s\in\mathcal{O}_2$ such that $(rs)^2$ has order $5$. Since the
centralizer corresponding to the class $\mathcal{O}$ has order $2^{11}$, the
claim follows from Lemma \ref{lem:usefulresult}.

\begin{step-b}
The classes \textup{32A, 32B, 32C, 32D} of $B$ are of type D.
\end{step-b}

Let $\mathcal{O}$ be the conjugacy class \textup{32A} of $B$.
We use the \gap~function \verb+PossibleClassFusions+ to obtain a list with all the
possible fusions from the maximal subgroup $\mathcal{M}_6$ into $B$.  As in the
previous step, with the \gap~functions \verb+ClassMultiplicationCoefficient+
and \verb+PowerMap+ it is easy to show that if $\mathcal{O}_1$ and
$\mathcal{O}_2$ are two different conjugacy classes of $\mathcal{M}_6$
contained in the class $\mathcal{O}$ of $B$, then there exist
$r\in\mathcal{O}_1$ and $s\in\mathcal{O}_2$ such that $(rs)^2$ has order
5. Since the centralizer related to the class $\mathcal{O}$ has size $2^7$, the
claim follows from Lemma \ref{lem:usefulresult}. The proof for the claim
concerning the classes \textup{32B, 32C, 32D} is analogous.
\end{proof}

\begin{rem}
	\label{rem:M}
	The following conjugacy classes of the Monster group $M$ are not known to be
	of type D: \textup{32A, 32B, 41A, 46A, 46B, 47A, 47B, 59A, 59B, 69A, 69B, 71A, 71B,
	87A, 87B, 92A, 92B, 94A, 94B}. These are the only open cases related to the problem
	of classifying conjugacy classes of type D in sporadic simple groups.
\end{rem}

\begin{rem}
	It is still unknown whether the Nichols algebras associated with the classes
	\textup{22A, 22B} of $Fi_{22}$, the classes \textup{2A, 46A, 46B} of
	$B$, and the classes \textup{32A, 32B, 46A, 46B, 92A, 92B, 94A, 94B} of $M$ are finite-dimensional.
\end{rem}

\subsection*{Acknowledgement}
First, we would like to thank the referee for his/her comments, which
allowed us to improve the original version of the paper, and for kindly communicating us the proof of Proposition \ref{pro:Fi23}.  We thank N.
Andruskiewitsch for interesting conversations and T. Breuer for Lemma
\ref{lem:Breuer_general} and Proposition \ref{pro:Fi22_22A22B}. We also thank
Facultad de Matem\'atica, Astronom\'{\i}a y F\'{\i}sica (Universidad Nacional
de C\'ordoba) for their computers \textsf{shiva} and \textsf{ganesh} where we
have performed the computations.  Part of the work of the first author was done
during a postdoctoral position in Universit\'e Paris Diderot -- Paris 7; he is
very grateful to Prof. Marc Rosso for his kind hospitality.


\begin{thebibliography}{AFGV11b}

\bibitem[AFGaV]{CLA}
N.~{Andruskiewitsch}, F.~{Fantino}, G.~A. {Garcia}, and L.~{Vendramin},
\emph{{On Nichols algebras associated to simple racks}}, Contemp. Math. \textbf{537}
  (2011) 31--56.

\bibitem[AFGV1]{AMPA}
N.~{Andruskiewitsch}, F.~{Fantino}, M.~{Gra{\~n}a}, and L.~{Vendramin},
  \emph{{Finite-dimensional pointed Hopf algebras with alternating groups are
  trivial}}, Ann. Mat. Pura Appl. (4) \textbf{190} (2011), no.~2, 225--245.

\bibitem[AFGV2]{MR2745541}
\bysame, \emph{The
  logbook of {P}ointed {H}opf algebras over the sporadic simple groups}, J.
  Algebra \textbf{325} (2011), 282--304.

\bibitem[AFGV3]{MR2745542}
\bysame, \emph{Pointed {H}opf algebras over the sporadic simple groups}, J.
  Algebra \textbf{325} (2011), 305--320.

\bibitem[AG]{MR1994219}
N.~Andruskiewitsch and M.~Gra{\~n}a, \emph{From racks to
  pointed {H}opf algebras}, Adv. Math. \textbf{178} (2003), no.~2, 177--243.

\bibitem[AHS]{MR2766176}
N.~Andruskiewitsch, I.~Heckenberger, and H.-J.~Schneider,
\emph{The {N}ichols algebra of a semisimple {Y}etter-{D}rinfeld
  module}, Amer. J. Math. \textbf{132} (2010), no.~6, 1493--1547.

\bibitem[AS1]{MR1659895}
N.~Andruskiewitsch and H.-J. Schneider, \emph{Lifting of quantum linear spaces
  and pointed {H}opf algebras of order {$p\sp 3$}}, J. Algebra \textbf{209}
  (1998), no.~2, 658--691.

\bibitem[AS2]{MR1913436}
\bysame,\emph{Pointed {H}opf
  algebras}, New directions in {H}opf algebras, Math. Sci. Res. Inst. Publ.,
  vol.~43, Cambridge Univ. Press, Cambridge, 2002, pp.~1--68.

\bibitem[B]{B}
{T. Breuer}, \emph{The GAP Character Table Library, Version 1.2};
\texttt{http://www.math.rwth-aachen.de/\~{}Thomas.Breuer/ctbllib/}

\bibitem[CCNPW]{MR827219}
J.~H. Conway, R.~T. Curtis, S.~P. Norton, R.~A. Parker, and R.~A. Wilson,
  \emph{Atlas of finite groups}, Oxford University Press, Eynsham, 1985,
  Maximal subgroups and ordinary characters for simple groups, With
  computational assistance from J. G. Thackray.

\bibitem[GAP]{GAP}
\emph{The {GAP}~{G}roup}, 2008, GAP -- Groups, Algorithms, and Programming,
  Version 4.4.12. Available at \verb+http://www.gap-system.org+.

\bibitem[HS]{MR2734956}
I.~Heckenberger and H.-J. Schneider, \emph{Root systems and {W}eyl groupoids
  for {N}ichols algebras}, Proc. Lond. Math. Soc. (3) \textbf{101} (2010),
  no.~3, 623--654.

\bibitem[I]{MR2270898}
M.~Isaacs, \emph{Character theory of finite groups}, AMS Chelsea
  Publishing, Providence, RI, 2006.
\bibitem[J]{MR682881}
D.~Joyce, \emph{Simple quandles}, J. Algebra \textbf{79} (1982), no.~2,
  307--318.

\bibitem[WPN+]{AtlasRep1.4}
R. A. Wilson, R. A. Parker, S. Nickerson,
J. N. Bray and T. Breuer, \emph{{AtlasRep}, A \textsf{GAP} Interface to the \textsf{ATLAS} of
Group Representations,
{V}ersion 1.4}, 2007, Refereed \textsf{GAP} package,
\verb+http://www.math.rwth-aachen.de/~Thomas.Breuer/atlasrep+.

\bibitem[WWT+]{ATLASwww}
R. A. Wilson, P. Walsh, J. Tripp, I. Suleiman,
R. Parker, S. Norton, S. Nickerson, S. Linton, J. Bray and R.
Abbott, \emph{A world-wide-web Atlas of finite group
representations},
\texttt{http://brauer.maths.qmul.ac.uk/Atlas/v3/}.

\end{thebibliography}
\end{document}